\documentclass{article}
\usepackage{amsmath,amsthm,amsfonts,amssymb,nicefrac}
\usepackage{setspace} 

\usepackage{url}

\usepackage[all]{xy}
\theoremstyle{plain}
\newtheorem{df}{Definition}
\newtheorem{tm}[df]{Theorem}
\newtheorem{cor}[df]{Corollary}

\newtheorem{lm}[df]{Lemma}
\newtheorem{pp}[df]{Proposition}

\theoremstyle{definition}

\begin{document}


\title{Numerical Semigroups on Compound Sequences}

\author{Claire Kiers, Christopher O'Neill, and Vadim Ponomarenko}
\date{}
\maketitle

\begin{abstract}
We generalize the geometric sequence $\{a^p, a^{p-1}b, a^{p-2}b^2,\ldots,b^p\}$ to allow the $p$ copies of $a$ (resp. $b$) to all be different.  We call the sequence $\{a_1a_2a_3\cdots a_p, b_1a_2a_3\cdots a_p, b_1b_2a_3\cdots a_p,\ldots, b_1b_2b_3\cdots b_p\}$ a \emph{compound sequence}.  We consider numerical semigroups whose minimal set of generators form a compound sequence, and compute various semigroup and arithmetical invariants, including the Frobenius number, Ap\'ery sets, Betti elements, and catenary degree.  We compute bounds on the delta set and the tame degree.
\end{abstract}

\noindent\emph{Keywords: nonunique factorization, numerical semigroup}\\
\emph{Mathematics Subject Classification (2010): 20M13, 11A51, 20M14 }



\section{Introduction}
\label{s:intro}

 Let $\mathbb{N}$ denote the set of positive integers, and $\mathbb{N}_0$ denote the set of nonnegative integers.  We call $S$ a \emph{numerical semigroup} if  $S\subseteq \mathbb{N}_0$, $S$ is closed under addition, $S$ contains $0$, and $|\mathbb{N}\setminus S|<\infty$. We say $\{x_0,x_1,\ldots, x_p\}$ is a \emph{set of generators} for $S$ if $S=\{\sum_{i=0}^p a_ix_i:a_i\in \mathbb{N}_0\}$, and call it minimal  if it is minimal as ordered by inclusion.  In this case we say $S$ has \emph{embedding dimension} $p+1$. For a general introduction to numerical semigroups, please see the monograph \cite{MR2549780}.

Numerical semigroups whose minimal generators are geometric sequences $\langle a^p, a^{p-1}b, a^{p-2}b^2,\ldots, b^p\rangle$ have been investigated recently in \cite{MR2425631,MR2472061}.  We propose a generalization of such sequences, which we call \emph{compound sequences}. These also generalize supersymmetric numerical semigroups, as defined in \cite{MR880351}, whose minimal generators are $\langle \frac{s}{t_1}, \frac{s}{t_2},\ldots, \frac{s}{t_n}\rangle$, where $s=t_1t_2\cdots t_n$, and the $t_i$ are pairwise coprime.

\begin{df}
Let $p,a_1, a_2, \ldots, a_p,b_1,b_2,\ldots, b_p\in \mathbb{N}$.  Suppose that:
\begin{enumerate}
\item $2\le a_i<b_i$, for each $i\in[1,p]$.
\item $\gcd(a_i,b_j)=1$ for all $i,j\in [1,p]$ with $i\ge j$.
\end{enumerate}
For each $i\in[0,p]$, we set $n_i=b_1b_2\cdots b_i a_{i+1}a_{i+2}\cdots a_p$. We then call the sequence $\{n_0,n_1,\ldots,n_p\}$ a \emph{compound sequence}.
\end{df}

From this definition it is clear that $\gcd(a_i,b_1b_2\cdots b_i)=1$,  $\gcd(a_ia_{i+1}\cdots a_p,b_i)=1$, and lastly $\gcd(a_ia_{i+1}\cdots a_p,b_1b_2\cdots b_i)=1$.  Note that the special case of $a_1=a_2=\cdots=a_p$, $b_1=b_2=\cdots=b_p$ gives a geometric sequence. 

Henceforth we will focus on numerical semigroups on compound sequences, which we will abbreviate as NSCS.  Such semigroups, though rare, are common enough to warrant study.  For example, consider numerical semigroups of embedding dimension 3, whose largest generator is at most 200.  Of these, 1\% have their generators in a compound sequence, while 0.6\% have their generators in an arithmetic sequence.  The latter class of semigroups, and variations thereof, has been the subject of much recent study in \cite{MR2269412, MR2494887, MR3245842,  MR2510984, MR2926638}, and many factorization invariants have been computed.  This paper does similarly for NSCS.

Given a numerical semigroup $S$ minimally generated by $n_0,\ldots, n_p$, the~map
\[\phi:\mathbb{N}_0^{p+1}\to S,~~~ \phi(x_0, x_1, \ldots, x_p)=x_0n_0+x_1n_1+\cdots +x_pn_p\] is a monoid homomorphism, called the \emph{factorization homomorphism} of $S$.  Let $\sigma$ be its kernel congruence, that is $x\sigma y$ if and only if $\phi(x)=\phi(y)$.  Then $S$ is isomorphic to $\mathbb{N}_0^{p+1}/\sigma$. We will consider $\sigma$ as a subset of $\mathbb{N}_0^{p+1}\times \mathbb{N}_0^{p+1}$.    Set $\mathcal{I}(S)$ to be the irreducibles of $\sigma$, viewed as a  monoid. Set $e_i$, for $i\in [0,p]$, to be the standard basis vectors of $\mathbb{N}_0^{p+1}$.  For $n\in S$, the set $\phi^{-1}(n)$ is the set of \emph{factorizations} of $n$.  We say $n>1$ is a \emph{Betti element} if there is a partition $\phi^{-1}(n)=X\cup Y$ satisfying $\sum_{i=0}^px_iy_i=0$ for each $x\in X, y\in Y$.  

Betti elements capture important semigroup structure, and have received considerable recent attention (\cite{MR3040913,MR3007913,MR3040805,MR2734166}).  In Section~\ref{s:factorization}, we examine the congruence $\sigma$ when $S$ is an NSCS.  We prove that any NSCS is both free and a complete intersection (Corollary~\ref{free}) by explicitly describing in Theorem~\ref{delta-chain} its unique minimal presentation (a minimal generating set of $\mathcal I(S)$).  As a consequence, we characterize the Betti elements of any NSCS (Corollary~\ref{Betti}).  

In Section~\ref{s:invariants}, we will compute for NSCS several arithmetic properties used in factorization theory.  For a general reference on factorization theory, see any of \cite{MR2571193,acm-survey,MR2194494}, and for more background on arithmetic invariants in general numerical semigroups see \cite{MR2494887,MR2243561}.  We now define the invariants considered in this paper, including the delta set (Corollary~\ref{delta-cor}), catenary degree (Theorem~\ref{catenary}) and tame degree (Theorem~\ref{tame-degree}).  

Fix a numerical semigroup $S$ and $n\in S$.  If $x=(x_0,\ldots, x_p)\in \phi^{-1}(n)$, the \emph{length} of the factorization $x$ is $|x|=x_0+\cdots +x_p$.  
We define the \emph{length set} of $n$ as $\mathcal{L}(n)=\{|x|:x\in \phi^{-1}(n)\}$.  Writing $\mathcal{L}(n)=\{s_1 < s_2 < \cdots < s_k\}$, we define the \emph{delta set} of $n$ as $\Delta(n)=\{s_i-s_{i-1}:i\in[2,k]\}$, with $\Delta(n)=\emptyset$ if $|\mathcal{L}(n)|=1$.  We define the \emph{delta set} of $S$ as $\Delta(S)=\cup_{n\in S}\Delta(n)$.  

For $x,y\in \mathbb{N}_0^{p+1}$, define $\gcd(x,y)$ and the distance $d(x,y)$ between by 
$$\begin{array}{rcl}
\gcd(x,y) &=& (\min\{x_0,y_0\}, \min\{x_1,y_1\},\ldots \min\{x_p,y_p\})\in \mathbb{N}_0^{p+1}, \\
d(x,y) &=& \max\left\{ |x-\gcd(x,y)|,|y-\gcd(x,y)| \right\}.
\end{array}$$
Further, for $Y\subseteq \mathbb{N}_0^{p+1}$, we define $d(x,Y)=\min\{d(x,y):y\in Y\}$.  Given $n\in S$ and $x,y\in \phi^{-1}(n)$, then a \emph{chain of factorizations} from $x$ to $y$ is a sequence $x^0, x^1, \ldots x^k\in \phi^{-1}(n)$ such that $x^0=x$ and $x^k=y$.  We call this an \emph{$N$-chain} if $d(x^i,x^{i+1})\le N$ for all $i\in [0,k-1]$.  The \emph{catenary degree} of $n$, $c(n)$, is the minimal $N\in\mathbb{N}_0$  such that for any two factorizations $x,y\in\phi^{-1}(n)$, there is an $N$-chain from $x$ to $y$.  The catenary degree of $S$, $c(S)$, is defined by \[c(S)=\sup\{c(n):n\in S\}.\]
For $i\in[0,p]$, we define $\phi^{-1}_i(n)=\{(x_0,\ldots, x_p)\in \phi^{-1}(n):x_i>0\}$.  We define $t_i(n)=\max\{d(z,\phi^{-1}_i(n)):z\in \phi^{-1}(n)\}$ for $\phi^{-1}_i(n)\neq \emptyset$, and set $t_i(n)=-\infty$ otherwise.  We define the \emph{tame degree} of $n$ as $t(n)=\max\{t_i(n):i\in[0,p]\}$, and the \emph{tame degree} of $S$ as $t(S)=\max\{t(n):n\in S\}$.  

We conclude this section by presenting some elementary properties of compound sequences to be used throughout the paper.

\begin{pp}\label{simple-properties}
Let $\{n_0,n_1,\ldots,n_p\}$ be a compound sequence as defined above.  Then the following all hold.
\begin{enumerate}
\item $n_i=\frac{b_i}{a_i}n_{i-1}$, for each $i\in[1,p]$.
\item $n_0<n_1<\ldots<n_p$.
\item $\gcd(n_0,n_1,\ldots, n_i)=\prod_{j=i+1}^p a_j$ for all $i\in [0,p]$.
\item $\gcd(n_i,n_{i+1},\ldots, n_p)=\prod_{j=1}^{i} b_j$ for all $i\in [0,p]$.
\item $\gcd(n_0,n_1,\ldots, n_p)=1$.  
\item $\langle n_0,n_1,\ldots,n_p\rangle$ is a minimally generated numerical semigroup.
\item $a_i=\frac{n_{i-1}}{\gcd(n_{i-1},n_i)}$ and $b_i=\frac{n_i}{\gcd(n_{i-1},n_i)}$, for each $i\in[1,p]$.
\end{enumerate}
\end{pp}
\begin{proof}
(1) trivial. ~~~ (2) follows from (1) since $\frac{n_i}{n_{i-1}}=\frac{b_i}{a_i}>1$ for each $i\in[1,p]$.  \\
(3) Set $A=\prod_{j=i+1}^p a_j$.  Since $A$ divides each of $n_0, \ldots, n_i$, it suffices to prove that $\gcd(n'_0,\ldots, n'_i)=1$, where $n'_0=\frac{n_0}{A}, \ldots, n'_i=\frac{n_i}{A}$.  Suppose prime $q$ divides $\gcd(n'_0,\ldots, n'_i)$.  Then $q|\gcd(n'_0,n'_i)=\gcd(a_1a_2\cdots a_i,b_1b_2\cdots b_i)$.  Let $k$ be maximal in $[1,i]$ so that $q|a_k$, and let $j$ be minimal in $[1,i]$ so that $q|b_j$. Since $q|\gcd(a_k,b_j)$, by the definition of compound sequences we must have $k<j$.  But now $q\nmid n'_k$, a contradiction.\\
(4) Similar to (3).  ~~ (5) follows from (3).\\
(6) This is a numerical semigroup by (5).  To prove it is minimally generated, we appeal to  Cor. 1.9 from \cite{MR2549780}, by which it suffices to prove that $n_{i}\notin \langle n_0,\ldots, n_{i-1}\rangle$ for each $i\in [1,p]$.  Set $x=a_ia_{i+1}\cdots a_p$.  We have $x|\gcd(n_0,n_1,\ldots, n_{i-1})$.  If $n_i\in\langle n_0,\ldots, n_{i-1}\rangle$ then $x|n_i=b_1b_2\cdots b_i a_{i+1}a_{i+2}\cdots a_p$.  Cancelling, we get $a_i|b_1b_2\cdots b_i$, a contradiction since $a_i>1$ yet $\gcd(a_i,b_1b_2\cdots b_i)=1$.\\
(7) follows by combining  $\gcd(n_{i-1},n_i)=b_1b_2\cdots b_{i-1}a_{i+1}a_{i+2}\cdots a_p\gcd(a_i,b_i)$ with $\gcd(a_i,b_i)=1$.
\end{proof}

Note that Proposition \ref{simple-properties}.7 suggests that the generators $n_0,\ldots n_p$ alone suffice to recover the $\{a_i\}, \{b_i\}$.  This is indeed the case, as shown in the following.

\begin{pp}\label{alternate-characterization}
Let $n_0, n_1,\ldots, n_p\in \mathbb{N}$ with $n_0<n_1<\cdots <n_p$.  Suppose  that  $\langle n_0, n_1,\ldots, n_p\rangle$ is a minimally generated numerical semigroup.  Then the following are equivalent.
\begin{enumerate}
\item $\{n_0,n_1,\ldots, n_p\}$ is a compound sequence.
\item $n_1n_2\cdots n_{p-1}= \gcd(n_0,n_1)\gcd(n_1,n_2)\cdots \gcd(n_{p-1},n_p)$
\end{enumerate}
\end{pp}
\begin{proof}
Applying Proposition~\ref{simple-properties}.7 to (1), we have $\frac{n_1}{\gcd(n_1,n_2)}\frac{n_2}{\gcd(n_2,n_3)}\cdots \frac{n_{p-1}}{\gcd(n_{p-1},n_p)}=a_2a_3\cdots a_p=\gcd(n_0,n_1)$, and cross-multiplying yields (2).  Conversely, define $a_i, b_i$ as in Proposition \ref{simple-properties}.7.  Note that $a_in_i=b_in_{i-1}$ and that $\gcd(a_i,b_i)=1$.  Also note that $a_i<b_i$ since $n_{i-1}<n_i$, and that $a_i>1$ since otherwise $n_{i-1}|n_i$ but the semigroup is minimally generated.  Dividing both sides of (2) by $\gcd(n_1,n_2)\cdots \gcd(n_{p-1},n_p)$, we get $a_2a_3\cdots a_p=\gcd(n_0,n_1)$.  Since $a_1=\frac{n_0}{\gcd(n_0,n_1)}$ we conclude that $n_0=a_1a_2\cdots a_p$.  Repeatedly applying $a_in_i=b_in_{i-1}$ gives $n_i=b_1b_2\cdots b_ia_{i+1}a_{i+2}\cdots a_p$ for $i\in [0,p]$.  Lastly, if $\gcd(a_i,b_1b_2\cdots b_i)=d>1$ for some $i$, then $d$ divides each of $n_0, n_1, \ldots, n_p$, a contradiction.
\end{proof}

Applying Proposition \ref{alternate-characterization}, we see that in embedding dimension 2, every numerical semigroup $\langle a,b\rangle$ is on a compound sequence.  Further, in embedding dimension 3, we see that numerical semigroup $\langle a,b,c\rangle$ is on a compound sequence if and only if we can write $b=b_1 b_2$ where $b_1|a$ and $b_2|c$.

\section{Factorization Structure}
\label{s:factorization}

We begin our study of NSCS by examining their factorizations.  These have very nice structure, which will be developed in this section.  We first compute their minimal presentation (Theorem~\ref{delta-chain}) by showing that any two factorizations of the same element are connected by a chain of basic swaps (Definition~\ref{basicswap}).

For nonzero $x\in \mathbb{Z}^{p+1}$, we define $\min(x)=\min\{i:x_i\neq 0\}$ and $\max(x)=\max\{i:x_i\neq 0\}$.  Note that for any $x,y\in \mathbb{N}_0^{p+1}$, $\min(x)\ge \min(x+y)$ and $\max(x)\le \max(x+y)$.  Note also that $\min(x-y)$ is the smallest coordinate where $x,y$ differ. This next, technical, result  divides factorizations of the important element $a_in_i=b_in_{i-1}$ into  two quite different categories.  In particular, it implies that they are each Betti elements.

\begin{pp}\label{left-or-right}
Let $S=\langle n_0,\ldots, n_p\rangle$ be an NSCS, and let $i\in [1,p]$.  Let $x\in\phi^{-1}(a_in_i)$.  Then one of the following must hold:
\begin{enumerate}
\item $\min(x)\ge i$ and $|x|\le a_i$; or
\item  $\max(x)\le i-1$ and $|x|\ge b_i$.
\end{enumerate}
Further, factorizations of both types exist, where all inequalities are met.
\end{pp}

\begin{proof}
Set $a=a_ia_{i+1}\cdots a_p, b=b_1b_2\cdots b_i$.  Note that  $a_in_i=ab$ and that $a$ divides each of $n_0,n_1,\ldots, n_{i-1}$ while $b$ divides each of $n_i, n_{i+1}, \ldots, n_p$.  We  have $$0\equiv a_in_i\equiv \sum_{j=0}^px_jn_j\equiv \sum_{j=0}^{i-1}x_jn_j\pmod{b}.$$  We divide both sides by $a$ (since $\gcd(a,b)=1$) to get $0\equiv \sum_{j=0}^{i-1}x_j\frac{n_j}{a}\pmod{b}$. If $ \sum_{j=0}^{i-1}x_j\frac{n_j}{a}=0$, then $\min(x)\ge i$.  Otherwise, $ b\le\sum_{j=0}^{i-1}x_j\frac{n_j}{a}$ and we multiply both sides by $a$ to get $ab\le \sum_{j=0}^{i-1}x_jn_j\le \sum_{j=0}^px_jn_j=a_in_i=ab$.  All the inequalities are equalities and hence $\max(x)\le i-1$.

Now, partition $\phi^{-1}(n)=X\cup Y$, where factorizations $x\in X$ satisfy $\min(x)\ge i$ and factorizations $y\in Y$ satisfy $\max(y)\le i-1$.  For any $x\in X$, we have $n=a_in_i=\sum_{j=i}^px_jn_j\ge |x|n_i$, and hence $|x|\le a_i$.  Similarly, for any $y\in Y$, we have $n=b_in_{i-1}=\sum_{j=1}^{i-1}y_jn_j\le |y|n_{i-1}$, and hence $|y|\ge b_i$.

Finally, note that $a_ie_i\in X$ and $b_ie_{i-1}\in Y$.
\end{proof}

This next lemma is essential for the proof of Theorem \ref{delta-chain}, and relates two factorizations of the same element in an NSCS, on their extremal coordinates.  We omit its straightforward proof.

\begin{lm}\label{modulus-lemma}
Fix an NSCS $S$, $n \in S$, and $x,y\in \phi^{-1}(n)$.  Set $m=\min(x+y)$ and $m'=\max(x+y)$.  Then $x_m\equiv y_m\pmod{ b_{m+1}}$ and  $x_{m'}\equiv y_{m'}\pmod{a_{m'}}$.
\end{lm}





\begin{df} \label{basicswap}
Fix an NSCS $S = \langle n_0,\ldots, n_p\rangle$
  A \emph{basic swap} is an element of the kernel congruence $\sigma$, for each $i\in [1,p]$, as given by \[\delta_i=(a_ie_i,b_ie_{i-1}), ~~ \delta_i'=(b_ie_{i-1},a_ie_i)\]

We  define $\Omega=\{\delta_i\}\cup \{\delta_i'\}$, as the set of all basic swaps.  For $\tau=(\tau_1,\tau_2)\in \Omega$, if $x+\tau_1=y+\tau_2$, we say that we \emph{apply the basic swap $\tau$ to get from $x$ to $y$}.  If $x^0,x^1,\ldots, x^k$ is a chain of factorizations in $\phi^{-1}(n)$, we call this a \emph{basic chain} if for each $i\in [1,k-1]$ we get from $x^{i+1}$ to $x^i$ by applying $\tau_i\in \Omega$. If a basic chain also satisfies, for all $i\in[1,k-1]$, that $\tau_i\in\{\delta_j, \delta_j'\}$, where $j=1+\min(x_{i-1}-x_i)$, we call it a \emph{left-first basic chain}.  Similarly, if  a basic chain also satisfies, for all $i\in[1,k-1]$, that $\tau_i\in\{\delta_j,\delta_j'\}$, where $j=\max(x_{i-1}-x_i)$, we call it a \emph{right-first basic chain}. 
\end{df} 

Note that if $z$ is part of either a left-first or right-first basic chain from $x$ to $y$, then $\min(z)\ge \min(x+y)$ and $\max(z)\le \max(x+y)$.  Note also that if we apply basic swap $\delta_i$ (or $\delta_i'$) to get from $x$ to $y$, then $d(x,y)=d(a_ie_i,b_ie_{i-1})=b_i$.  Each basic swap is in $\sigma$ since $a_in_i=b_in_{i-1}$, but in fact basic swaps are irreducibles in $\sigma$, as shown by the following.  

\begin{lm}\label{irreducible-lemma}
Let $S=\langle n_0,\ldots, n_p\rangle$ be an NSCS.  Then $\Omega\subseteq \mathcal{I}(\sigma)$.
\end{lm}
\begin{proof}
If some fixed $\delta_i$ were reducible, then there is some $(\alpha e_i, \beta e_{i-1})\in \sigma$, with $0<\alpha<a_i$.  Hence $\alpha b_1\cdots b_{i-1}b_ia_{i+1}\cdots a_P=\phi(\alpha e_i)=\phi(\beta e_{i-1})=\beta b_1\cdots b_{i-1}a_ia_{i+1}\cdots a_P$.  Cancelling, we get $\alpha b_i=\beta a_i$ and hence $\alpha b_i\equiv 0\pmod{a_i}$.  Since $\gcd(a_i,b_i)=1$,  in fact $\alpha \equiv 0\pmod{a_i}$, a contradiction.
\end{proof}

 The following theorem proves the existence of basic chains connecting any two factorizations.  Combined with Lemma \ref{irreducible-lemma}, it implies that $\Omega$ is a minimal presentation of $\sigma$.  

\begin{tm}\label{delta-chain}
Fix an NSCS $S=\langle n_0,\ldots, n_p\rangle$ and $n \in S$.  For any $x,y\in \phi^{-1}(n)$, there are both left-first and right-first basic chains of factorizations from $x$ to $y$.
\end{tm}

\begin{proof}  We will only prove the existence of a left-first basic chain (the right-first case is similar).  We argue by way of contradiction.  Let $n$ be minimal possessing at least one pair of factorizations $x,y\in \phi^{-1}(n)$ that do not admit a left-first basic chain between them.  Of all such pairs in $\phi^{-1}(n)$ not admitting a basic chain, choose a pair $x,y\in\phi^{-1}(n)$  with $|x_{\min(x+y)}-y_{\min(x+y)}|$  minimal.  For convenience, set $t=\min(x+y)$.  Depending on whether $x_t-y_t$ is nonzero or zero, we now have two cases, each of which will lead to contradiction.

If $x_t-y_t$ is positive (resp. negative), we apply Lemma \ref{modulus-lemma}, and then apply $\delta_{t+1}$ to $x$ (resp. $y$) to get a new $z\in\phi^{-1}(n)$.  Applying the inductive hypothesis, we get a left-first basic chain of factorizations from $z$ to $y$ (resp. from $x$ to $z$).  But now we may extend this to a left-first basic chain from $x$ to $y$, which yields a contradiction.



Lastly we have  $x_t=y_t>0$.  We now set $\bar{n}=n-x_tn_t$, $\bar{x}=x-x_te_t$, $\bar{y}=y-y_te_t$.  Since $\bar{n}<n$, by the choice of $n$ any two factorizations of $\bar{n}$ must admit a left-first basic chain between them.  In particular, $\bar{x},\bar{y}\in\phi^{-1}(\bar{n})$ must admit a left-first basic chain $\bar{x}^0, \bar{x}^1,\ldots, \bar{x}^k$.  But then $(\bar{x}^0+x_te_t), (\bar{x}^1+x_te_t),\ldots, (\bar{x}^k+x_te_t)$ is a left-first basic chain from $x$ to $y$, which is a contradiction.
\end{proof}

We recall that a numerical semigroup is a complete intersection if the cardinality each of its minimal presentations is one less than its embedding dimension.  We recall that a numerical semigroup is free if for some ordering of its generators $n'_1,\ldots, n'_p$, and for all $i\in[2,p]$, we have $\min\{k\in\mathbb{N}:kn'_i\in \langle n'_1,\ldots, n'_{i-1}\rangle\}=\min\{k\in\mathbb{N}:kn'_i\in \langle n'_1,\ldots, n'_{i-1},n'_{i+1},\ldots,n'_p\rangle\}$.

\begin{cor}\label{free}
Let $S=\langle n_0,\ldots, n_p\rangle$ be an NSCS.  Then $S$ is a free numerical semigroup, and a complete intersection.
\end{cor}
\begin{proof}
Corollaries 8.17 and 8.19 of \cite{MR2549780}.
\end{proof}

\begin{cor}\label{Betti}
Let $S=\langle n_0,\ldots, n_p\rangle$ be an NSCS.  Then $\{a_1n_1,a_2n_2,\ldots, a_pn_p\}$ is the set of Betti elements of $S$.
\end{cor}

\section{Ap\'ery sets}
\label{s:apery}

For a semigroup $S$ and $m\in S$, recall that the Ap\'ery set is defined as 
\[Ap(S,m)=\{n\in S:n-m\notin S\}.\]
These are most commonly computed when $m$ is an irreducible.  In this section, we will compute these Ap\'ery sets when $S$ is an NSCS (Theorem~\ref{apery}), after introducing $i$-normal factorizations (Definition~\ref{inormal-def}).


\begin{df}\label{inormal-def}
For a fixed NSCS $S=\langle n_0,\ldots, n_p\rangle$, a fixed $n\in S$, and a fixed $i\in[0,p]$, we call factorization $x\in \phi^{-1}(n)$ \emph{$i$-normal} if it satisfies $0\le x_j<b_{j+1}$ for all $j<i$, and $0\le x_j<a_j$ for all $j>i$.
\end{df}

Note that these conditions are equivalent to none of the basic swaps in the set $\{\delta_1, \delta_2, \ldots, \delta_i, \delta_{i+1}',\delta_{i+2}',\ldots,\delta_p'\}$ applying to $x$.  The following proposition justifies calling the term ``normal''.

\begin{pp}\label{i-normal}
Let $S=\langle n_0,\ldots, n_p\rangle$ be an NSCS.  Let $n\in S$, and let $i\in [0,p]$.  Then there is exactly one $x\in \phi^{-1}(n)$ that is $i$-normal.
\end{pp}
\begin{proof}We begin with an arbitrary factorization in $\phi^{-1}(n)$, and apply the following algorithm.    In each of the $p$ steps, we change our factorization to another, that is a bit closer to $i$-normal.  Each step corresponds to a basic swap from the ordered list $\delta_1,\delta_2,\ldots, \delta_i,\delta_p', \delta_{p+1}',\ldots, \delta_{i+1}'$.  In each step, we apply the basic swap as many times as possible, while still retaining a factorization of $n$.  This will decrease a coordinate to satisfy the conditions of $i$-normality.  Note that the list is ordered so that after a coordinate is decreased, it is never increased again.   Hence the algorithm terminates with an $i$-normal factorization.

 

We now prove uniqueness.  Let $x,y$ be $i$-normal factorizations of $n$.  Set $s=\min(x-y)$.  Suppose that  $s<i$.  We set $z=(x_0,x_1,\ldots, x_{s-1},0,0,\ldots, 0)$, and  apply Lemma \ref{modulus-lemma} to $x-z, y-z$, both factorizations of $n-\phi(z)$.  We conclude that $x_s\equiv y_s \pmod{b_{s+1}}$; however since $x,y$ are $i$-normal in fact $x_s=y_s$, a contradiction.  Hence $\min(x-y)\ge i$. Similarly, $\max(x-i)\le i$. 
Hence $x,y$ agree, except possibly for $x_i, y_i$.  However if $x_i\neq y_i$ they would not be factorizations of the same $n$.
\end{proof}

Uniqueness in Proposition~\ref{i-normal} yields several consequences.  Our first observation is that $i$-normal factorizations are maximal in the $i$-th coordinate.

\begin{cor}\label{normal-maximal} Let $S=\langle n_0,\ldots, n_p\rangle$ be an NSCS.  Let $n\in S$, and let $i\in [0,p]$.   Let $x,y\in\phi^{-1}(s)$, and suppose that $x$ is $i$-normal.  Then $x_i\ge y_i$.
\end{cor}
\begin{proof}
Suppose that $y_i>x_i$.  Then $n-\phi(y_in_i)\in S$, and has an $i$-normal factorization $z$. But now $z+y_ie_i$ is an $i$-normal factorization for $n$, which contradicts the uniqueness of $x$.
\end{proof}

Note that since $a_i<b_i$, applying any basic swap $\delta_i$ decreases the factorization length, while applying any $\delta_i'$ increases the factorization length.  This observation, together with Theorem \ref{delta-chain} and the comments preceding Proposition \ref{i-normal}, yield the following.

\begin{cor}\label{max-min}
Let $S=\langle n_0,\ldots, n_p\rangle$ be an NSCS.  Let $n\in S$.  Then the minimum factorization length of $n$ is the length of the $p$-normal factorization of $n$.  Also, the maximum factorization length of $n$ is the length of the $0$-normal factorization of $n$.
\end{cor}

We are now ready to compute the Ap\'ery set of $S$.  For $n \in S$, we let $x$ be the $i$-normal factorization for $n$.  The following theorem proves that $n\in Ap(S,n_i)$ if and only if $x_i=0$. 

\begin{tm}\label{apery} Let $S=\langle n_0,\ldots, n_p\rangle$ be an NSCS.  Let $i\in[0,p]$.  Then the Ap\'ery set $Ap(S,n_i)=\{\phi(u):u\in S_i\}$, where \[S_i=\left\{u\in \mathbb{N}_0^{p+1}: \begin{minipage}{6cm}
$u_0< b_1, u_1<b_2,\ldots,  u_{i-1}<b_i, u_i=0,$\\$u_{i+1}<a_{i+1},u_{i+2}<a_{i+2}, \ldots,  u_p<a_p$ \end{minipage}\right\}.\]

\end{tm}
\begin{proof}
If $x\in S_i$, then $x$ is $i$-normal, and hence by Corollary \ref{normal-maximal}, $x_i=0$ is maximal over all factorizations of $\phi(x)$.  Hence $\phi(x-x_i)\notin S$, and $\phi(x)\in Ap(S,n_i)$.  On the other hand, for $n\in Ap(S,n_i)$, let $x$ be the $i$-normal factorization of $n$.  If $x_i>0$ then $n-n_i\in S$, which is impossible.  Hence $x_i=0$ and thus $x\in S_i$.
\end{proof}

For a numerical semigroup $S$, recall that the largest integer in $\mathbb{N}\setminus S$ is called the \emph{Frobenius number} of $S$, denoted $g(S)$.  In \cite{MR0148606}, Brauer and Shockley observed that $g(S)=\max Ap(S;m)-m$.  Applying this to Theorem~\ref{apery}, with $i=0$ for simplicity, yields the followin direct generalization of the main result of \cite{MR2425631}.

\begin{cor}\label{frobenius}
Let $S=\langle n_0,\ldots, n_p\rangle$ be an NSCS.  Let $i\in[0,p]$.  Then \[g(S)=-n_0+\sum_{j=1}^pn_j(a_j-1).\]
\end{cor}

For a numerical semigroup $S$, recall that $|\mathbb{N}\setminus S|$ is called the \emph{genus of $S$} (\cite{MR0441855}).  
Corollaries~\ref{free} and~\ref{frobenius}, with $i=0$ for simplicity, imply the following.

\begin{cor}
Let $S=\langle n_0,\ldots, n_p\rangle$ be an NSCS.   Then \[|\mathbb{N}\setminus S|=\frac{1}{2}\left(1-n_0+\sum_{j=1}^pn_j(a_j-1)\right).\]
\end{cor}


\section{Arithmetic Invariants}
\label{s:invariants}

We now compute several arithmetic invariants in the NSCS context.  First we consider the catenary degree $c(S)$, which we can determine exactly.  In the special case of a geometric sequence $S=\langle a^p,a^{p-1}b,\ldots, b^p\rangle$, this gives $c(S)=b$.

\begin{tm}\label{catenary}
Let $S=\langle n_0,\ldots, n_p\rangle$ be an NSCS.  Then $c(S)=\max\{b_1,b_2,\ldots, b_p\}$.
\end{tm}
\begin{proof}
By Theorem \ref{delta-chain}, we may connect any two factorization by a basic chain.  Hence  $c(S)\le \max\{b_1,\ldots, b_p\}$.  Now fix $i$ such that $b_i=\max\{b_1,b_2,\ldots, b_p\}$.  Let $x,y$ be two factorizations of $a_in_i$ of the two types guaranteed by Proposition \ref{left-or-right}.  We have $\gcd(x,y)=0$ so $d(x,y)=\max\{|x|,|y|\}\ge b_i$.  Any chain of factorizations connecting $a_in_i$ to $b_in_{i-1}$ must at some point cross from one factorization type to the other, a step of size at least $b_i$.  Hence $c(a_in_i)\ge b_i$, so $c(S)\ge  \max\{b_1,\ldots, b_p\}$. 
\end{proof}

Note that in \cite{MR2243561} it was shown that the catenary degree of a numerical semigroup is achieved at a Betti element.  Theorem \ref{catenary} identifies the specific Betti elements, and the exact catenary degree.

We now consider $\Delta(S)$ in our context, which we can partially determine.  Recall from \cite{MR1100372} that $\min(\Delta(S))=\gcd(\Delta(S))$.  

\begin{tm}\label{delta-set}
Let $S=\langle n_0,\ldots, n_p\rangle$ be an NSCS.  Set $N=\{b_1-a_1,b_2-a_2,\ldots, b_p-a_p\}$.  Then:
\begin{enumerate}
\item $\min(\Delta(S))=\gcd(N)$,
\item $N\subseteq \Delta(S)$, and
\item $\max(\Delta(S))= \max(N)$.
\end{enumerate}
\end{tm}
\begin{proof} 
(1) Each basic swap is irreducible in $\sigma$.  Hence by Proposition 2.2 of \cite{MR2269412}, $\min(\Delta(S))\le \gcd(N)$.  For the reverse direction, note that  $n_i-n_{i-1}=(b_i-a_i)b_1\cdots b_{i-1}a_{i+1}\cdots a_p$, so $\gcd(N)|\gcd(\{n_i-n_{i-1}:i\in[1,p]\})$, which equals $\min(\Delta(S))$ by Proposition 2.10 of \cite{MR2269412}. \\
(2) The elements of $N$ correspond to factorization length changes in Betti elements, which must be in $\Delta(S)$.\\
(3) By Theorem 2.5 of \cite{MR3040913}, the largest element of $\Delta(S)$ arises from a factorization length change of a Betti element.
\end{proof}

In certain cases, Theorem \ref{delta-set} determines $\Delta(S)$ completely.  In particular, the geometric sequence case is settled since that restriction implies $|N|=1$. 

\begin{cor}\label{delta-cor}
Let $S=\langle n_0,\ldots, n_p\rangle$ be an NSCS.  Set $N=\{b_1-a_1,b_2-a_2,\ldots, b_n-a_n\}$.   Suppose that any of the following hold:
\begin{enumerate}
\item $|N|=1$, or 
\item $|N|>1$ and for some $\alpha\in \mathbb{N}$, $N=\{\alpha, 2\alpha, \ldots, |N|\alpha\}$, or
\item $|N|>1$ and for some $\alpha\in \mathbb{N}$, $N=\{2\alpha, 3\alpha, \ldots, (|N|+1)\alpha\}$,
\end{enumerate}
Then $\Delta(S)$ is completely determined.  In the first two cases, $\Delta(S)=N$; in the last case $\Delta(S)=N\cup \{\alpha\}$.
\end{cor}

For example, consider the NSCS given by $a_1=a_2=7, b_1=17, b_2=22$, i.e. $S=\langle 49,119,374\rangle$.  We have $N=\{10,15\}$, so applying Corollary \ref{delta-cor} gives $\Delta(S)=\{5,10,15\}$.

Beyond Corollary \ref{delta-cor}, more work is needed to determine $\Delta(S)$.   For example, the NSCS $S=\langle 4,14,63\rangle$ given by $a_1=a_2=2, b_1=7, b_2=9$ has $N=\{5,7\}$.  Applying Theorem \ref{delta-set} gives us $\{1,5,7\}\subseteq \Delta(S)$, while a computation with the GAP numericalsgps package (see \cite{numerical-package}) shows that $\Delta(S)=\{1,2,3,5,7\}$.

Lastly, we consider the tame degree $t(S)$.  We now prove two lower bounds for $t(S)$.  They arise by considering the smallest $r_p$ such that $r_pa_pn_p-n_0\in S$, and the smallest $s_1$ such that $s_1b_1n_0-n_p\in S$.

\begin{tm}\label{tame-degree}
Let $S=\langle n_0,\ldots, n_p\rangle$ be an NSCS.   Set $r_1=1$ and $r_i=\lceil \frac{a_{i-1}r_{i-1}}{b_i}\rceil$ for $i\in[2,p]$. Set $s_p=1, s_{i-1}=\lceil \frac{b_is_i}{a_{i-1}}\rceil$ for $i\in[2,p]$. Then 
\[t(S)\ge \max\{(b_1-a_1)r_1+(b_2-a_2)r_2+\cdots+(b_{p-1}-a_{p-1})r_{p-1}+b_pr_p,b_1s_1\}.\]
\end{tm}
\begin{proof} For the first bound, we set $n=a_pr_pn_p$ and $z=a_pr_pe_p\in \phi^{-1}(n)$.  We now set $u=(b_1r_1, b_2r_2-a_1r_1, b_3r_3-a_2r_2,\ldots, b_pr_p-a_{p-1}r_{p-1},0)$. Note the left-first basic chain $$u\to\overset{r_1\delta_2}{\cdots}\to(0,b_2r_2,b_3r_3-a_2r_2,\ldots, b_pr_p-a_{p-1}r_{p-1},0)\to\overset{r_2\delta_3}{\cdots}\to$$
$$\to (0,0,b_3r_3,\ldots, b_pr_p-a_{p-1}r_{p-1},0)\to\cdots\to b_pr_pe_{p-1}\to\overset{r_p\delta_p}{\cdots}\to z.$$ In particular $u\in\phi^{-1}_0(n)$.  Note that each $w\in \phi^{-1}_0(n)$ has $|w|>|z|$ and hence $d(w,z)=|w|$.  We will now show by way of Corollary \ref{max-min} that $|u|\le |w|$  for all $w\in\phi^{-1}_0(n)$. First, by Lemma \ref{modulus-lemma}, $w_0\ge b_1r_1=u_0$. Now, for all $i\in[1,p-1]$ we must have $u_i<b_{i+1}$ since otherwise $b_{i+1}\lceil \frac{a_ir_i}{b_{i+1}}\rceil-a_ir_i\ge b_{i+1}$, a contradiction.  Hence $u-b_1r_1e_1$ is a $p$-normal factorization.  By Corollary \ref{max-min}, $|u-b_1r_1e_0|\le |w-b_1r_1e_0|$ and hence $|u|\le |w|$.  Therefore  $t_0(n)\ge d(z,\phi^{-1}_0(n))=d(z,u)=|u|=(b_1-a_1)r_1+(b_2-a_2)r_2+\cdots+(b_{p-1}-a_{p-1})r_{p-1}+b_pr_p$.

For the second bound, we set $n=b_1s_1n_0$ and $z=b_1s_1e_0\in \phi^{-1}(n)$. We now set $u=(0,a_1s_1-b_2s_2, a_2s_2-b_3s_3,\ldots, a_{p-1}s_{p-1}-b_ps_p,a_ps_p)$. Note the right-first basic chain $$u\to\overset{s_p\delta'_p}{\cdots}\to(0,a_1s_1-b_2s_2, a_2s_2-b_3s_3,\ldots, a_{p-1}s_{p-1},0)\to\cdots\to$$ $$\to a_1s_1e_1\to\overset{s_1\delta'_1}{\cdots}\to z.$$ In particular $u\in\phi^{-1}_p(n)$.  Note that, by Corollary \ref{max-min}, each $w\in \phi^{-1}(n)$ has $|w|\le|z|$.  First, by Lemma \ref{modulus-lemma}, $w_p\ge a_ps_p=u_p$.
For all $i\in[1,p-1]$ we must have $u_i<a_i$ since otherwise $a_i\lceil \frac{b_{i+1}s_{i+1}}{a_i}\rceil-b_{i+1}s_{i+1}\ge a_i$, a contradiction.  Hence $u-a_ps_pe_p$ is a $0$-normal factorization.  We apply Corollary \ref{normal-maximal} to conclude that since $u_0=0$, also $w'_0=0$ for all $w'\in \phi^{-1}(n-a_ps_pn_p)$.  Therefore $w_0=0$ for all $w\in \phi^{-1}_i(n)$, and hence $d(z,\phi^{-1}_i(n))=|z|=b_1s_1$, as desired.
\end{proof}

We have no examples where this inequality is strict.  The following examples show that both parts of the bound are necessary.  For $S=\langle 165,176,208\rangle$, we compute $t(S)=27$ while Theorem \ref{tame-degree} gives $t(S)\ge \max\{27,16\}$.  For $S=\langle 165,195,208\rangle$, we compute $t(S)=26$ while Theorem \ref{tame-degree} gives $t(S)\ge \max\{18,26\}$.

\bibliographystyle{plain}
\bibliography{acm}

\end{document}